\documentclass[12pt]{amsart}
\usepackage{hyperref} 
\usepackage{amsmath, amsthm, amssymb}
\usepackage{hyperref} 
\usepackage{enumerate}
\usepackage{verbatim}
\usepackage{caption}
\usepackage{subcaption}
\usepackage{esint}
\usepackage[T1]{fontenc}

\usepackage{tikz,amsthm,amsmath,amstext,amssymb,amscd,epsfig,euscript, mathrsfs,
 dsfont,pspicture,
multicol,graphpap,graphics,graphicx,times,enumerate,
sidecap,
wrapfig,color,pict2e}
\usepackage{setspace}

\numberwithin{equation}{section}

\title{Quantitative straightening of distance spheres}
\date{\today}
\author{Guy C. David}
\address{Department of Mathematical Sciences\\ Ball State University, Muncie, IN 47306}
\email{gcdavid@bsu.edu}
\author{McKenna Kaczanowski}
\address{Department of Mathematical Sciences\\ Ball State University, Muncie, IN 47306}
\email{mskaczanowski@outlook.com}
\author{Dallas Pinkerton}
\address{Department of Mathematical Sciences\\ Ball State University, Muncie, IN 47306}
\email{dcpinkerton@bsu.edu}

\thanks{G.~ C.~ David was partially supported by the National Science Foundation under Grants No. DMS-1758709 and DMS-2054004.}

\begin{document}

\theoremstyle{plain}
\newtheorem{theorem}{Theorem}
\newtheorem{exercise}{Exercise}
\newtheorem{corollary}[theorem]{Corollary}
\newtheorem{scholium}[theorem]{Scholium}
\newtheorem{claim}[theorem]{Claim}
\newtheorem{lemma}[theorem]{Lemma}
\newtheorem{sublemma}[theorem]{Lemma}
\newtheorem{proposition}[theorem]{Proposition}
\newtheorem{conjecture}[theorem]{Conjecture}

\theoremstyle{definition}
\newtheorem{fact}[theorem]{Fact}
\newtheorem{example}[theorem]{Example}
\newtheorem{definition}[theorem]{Definition}
\newtheorem{remark}[theorem]{Remark}
\newtheorem{question}[theorem]{Question}

\numberwithin{equation}{section}
\numberwithin{theorem}{section}

\newcommand{\cG}{\mathcal{G}}
\newcommand{\RR}{\mathbb{R}}
\newcommand{\HH}{\mathcal{H}}
\newcommand{\LIP}{\textnormal{LIP}}
\newcommand{\Lip}{\textnormal{Lip}}
\newcommand{\Tan}{\textnormal{Tan}}
\newcommand{\length}{\textnormal{length}}
\newcommand{\dist}{\textnormal{dist}}
\newcommand{\diam}{\textnormal{diam}}
\newcommand{\vol}{\textnormal{vol}}
\newcommand{\rad}{\textnormal{rad}}
\newcommand{\side}{\textnormal{side}}

\def\bA{{\mathbb{A}}}
\def\bB{{\mathbb{B}}}
\def\bC{{\mathbb{C}}}
\def\bD{{\mathbb{D}}}
\def\bR{{\mathbb{R}}}
\def\bS{{\mathbb{S}}}
\def\bO{{\mathbb{O}}}
\def\bE{{\mathbb{E}}}
\def\bF{{\mathbb{F}}}
\def\bH{{\mathbb{H}}}
\def\bI{{\mathbb{I}}}
\def\bT{{\mathbb{T}}}
\def\bZ{{\mathbb{Z}}}
\def\bX{{\mathbb{X}}}
\def\bP{{\mathbb{P}}}
\def\bN{{\mathbb{N}}}
\def\bQ{{\mathbb{Q}}}
\def\bK{{\mathbb{K}}}
\def\bG{{\mathbb{G}}}

\def\nrj{{\mathcal{E}}}
\def\cA{{\mathscr{A}}}
\def\cB{{\mathscr{B}}}
\def\cC{{\mathscr{C}}}
\def\cD{{\mathscr{D}}}
\def\cE{{\mathscr{E}}}
\def\cF{{\mathscr{F}}}
\def\cB{{\mathscr{G}}}
\def\cH{{\mathscr{H}}}
\def\cI{{\mathscr{I}}}
\def\cJ{{\mathscr{J}}}
\def\cK{{\mathscr{K}}}
\def\Layer{{\rm Layer}}
\def\cM{{\mathscr{M}}}
\def\cN{{\mathscr{N}}}
\def\cO{{\mathscr{O}}}
\def\cP{{\mathscr{P}}}
\def\cQ{{\mathscr{Q}}}
\def\cR{{\mathscr{R}}}
\def\cS{{\mathscr{S}}}
\def\Up{{\rm Up}}
\def\cU{{\mathscr{U}}}
\def\cV{{\mathscr{V}}}
\def\cW{{\mathscr{W}}}
\def\cX{{\mathscr{X}}}
\def\cY{{\mathscr{Y}}}
\def\cZ{{\mathscr{Z}}}

  \def\del{\partial}
  \def\diam{{\rm diam}}
	\def\VV{{\mathcal{V}}}
	\def\FF{{\mathcal{F}}}
	\def\QQ{{\mathcal{Q}}}
	\def\BB{{\mathcal{B}}}
	\def\XX{{\mathcal{X}}}
	\def\PP{{\mathcal{P}}}

  \def\del{\partial}
  \def\diam{{\rm diam}}
	\def\image{{\rm Image}}
	\def\domain{{\rm Domain}}
  \def\dist{{\rm dist}}
	\newcommand{\Gr}{\mathbf{Gr}}
\newcommand{\md}{\textnormal{md}}
\newcommand{\vspan}{\textnormal{span}}

\begin{abstract}
    We study ``distance spheres'': the set of points lying at constant distance from a fixed arbitrary subset $K$ of $[0,1]^d$. We show that, away from the regions where $K$ is ``too dense'' and a set of small volume, we can decompose $[0,1]^d$ into a finite number of sets on which the distance spheres can be ``straightened'' into subsets of parallel $(d-1)$-dimensional planes by a bi-Lipschitz map. Importantly, the number of sets and the bi-Lipschitz constants are independent of the set $K$.
\end{abstract}

\maketitle

\section{Introduction}

Let $K$ be an arbitrary set in $\RR^d$ and $r\geq 0$. The set of all points whose distance from $K$ is equal to $r$ forms a new set that we call a ``distance sphere'', and denote $S_K(r)$. (A precise definition is given below; in fact, we will focus our attention on the unit cube of $\RR^d$ rather than the whole space.)

If $K$ consists of a single point, then $S_K(r)$ is simply the sphere of radius $r$ centered on $K$. If $K$ is a general set, the distance spheres may be rather complicated objects, whose structure may change wildly as $r$ varies. Figures \ref{fig:finite} and \ref{fig:cantor} below depict some examples. These sets have been studied (under different names) by many authors, e.g., \cite{Brown, Ferry, Fu, VellisWu}.

This paper is concerned with the geometric structure of distance spheres, from a quantitative perspective. Our goal is to find large subsets of $\RR^d$ on which all the distance spheres can be simultaneously ``straightened out'' into (subsets of) parallel $(d-1)$-dimensional planes by a global mapping with controlled distortion. Moreover, we control the number of subsets and the distortion of the ``straightening map'' by constants that depend on the dimension $d$ but are otherwise independent of the set $K$.

In order to accomplish this, we must ``throw away'' some pieces of the domain on which we cannot straighten the distance spheres. These pieces come in two types: one a piece of small $d$-dimensional volume, and one the union of all locations where the set $K$ is ``too dense''. These are defined precisely below, and our main theorem is then stated as Theorem \ref{thm:main}.

The main tools in our arguments are the results of \cite{AzzamSchul} and \cite{DavidSchul} for general Lipschitz functions, combined with an analysis of the ``mapping content'' defined in \cite{AzzamSchul} in the special case of the distance function $\dist(\cdot, K)$.

\subsection{Main definitions and results}

\begin{definition}
Let $K\subseteq [0,1]^d$ be a set. For $r\geq 0$, the \textit{distance spheres} for $K$ are the sets
$$ S_K(r) = \{x\in [0,1]^d: \dist(x,K)=r\}.$$
\end{definition}

\begin{figure}
     \centering
     \begin{subfigure}[b]{0.3\textwidth}
         \centering
         \includegraphics[width=\textwidth]{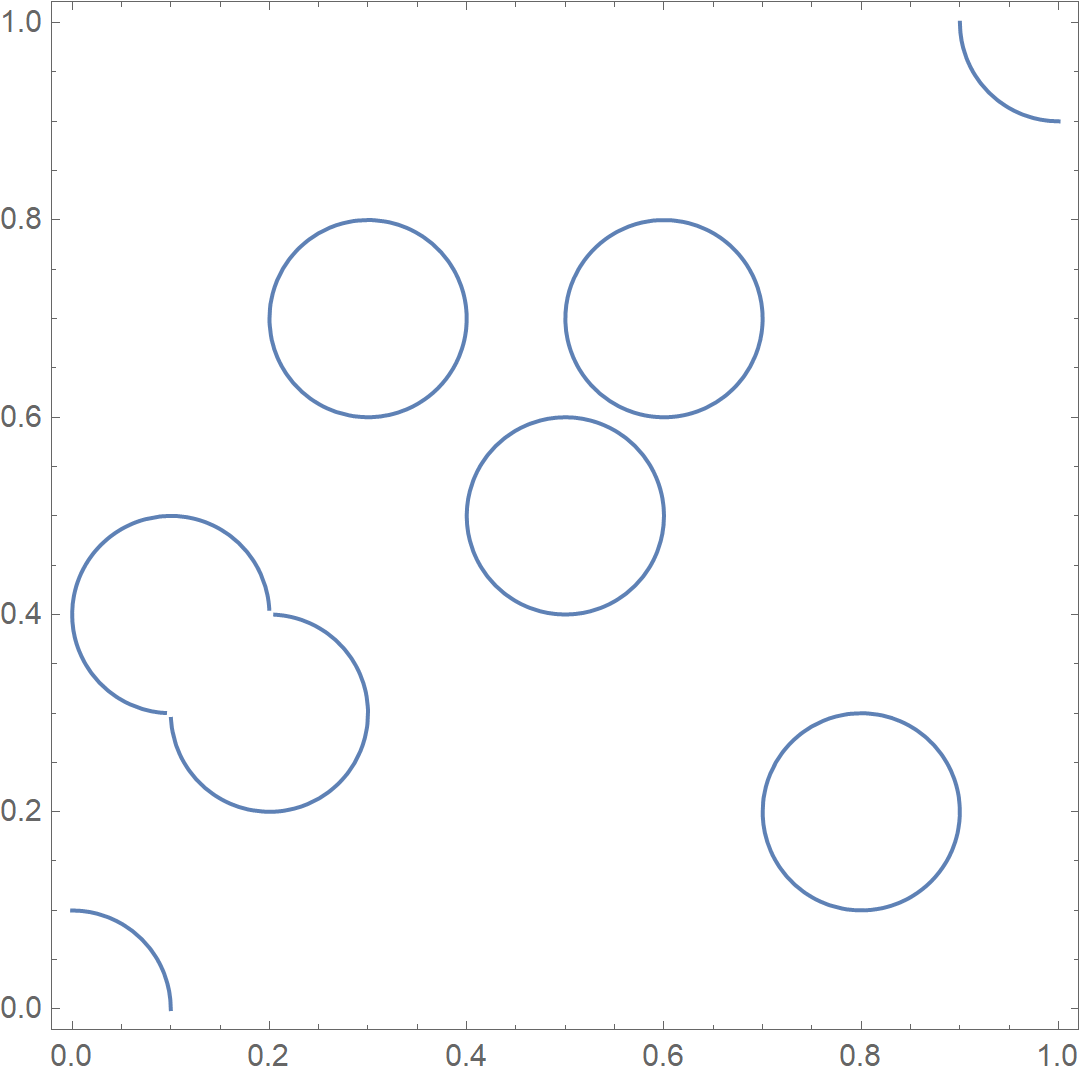}
     \end{subfigure}
     \hfill
     \begin{subfigure}[b]{0.3\textwidth}
         \centering
         \includegraphics[width=\textwidth]{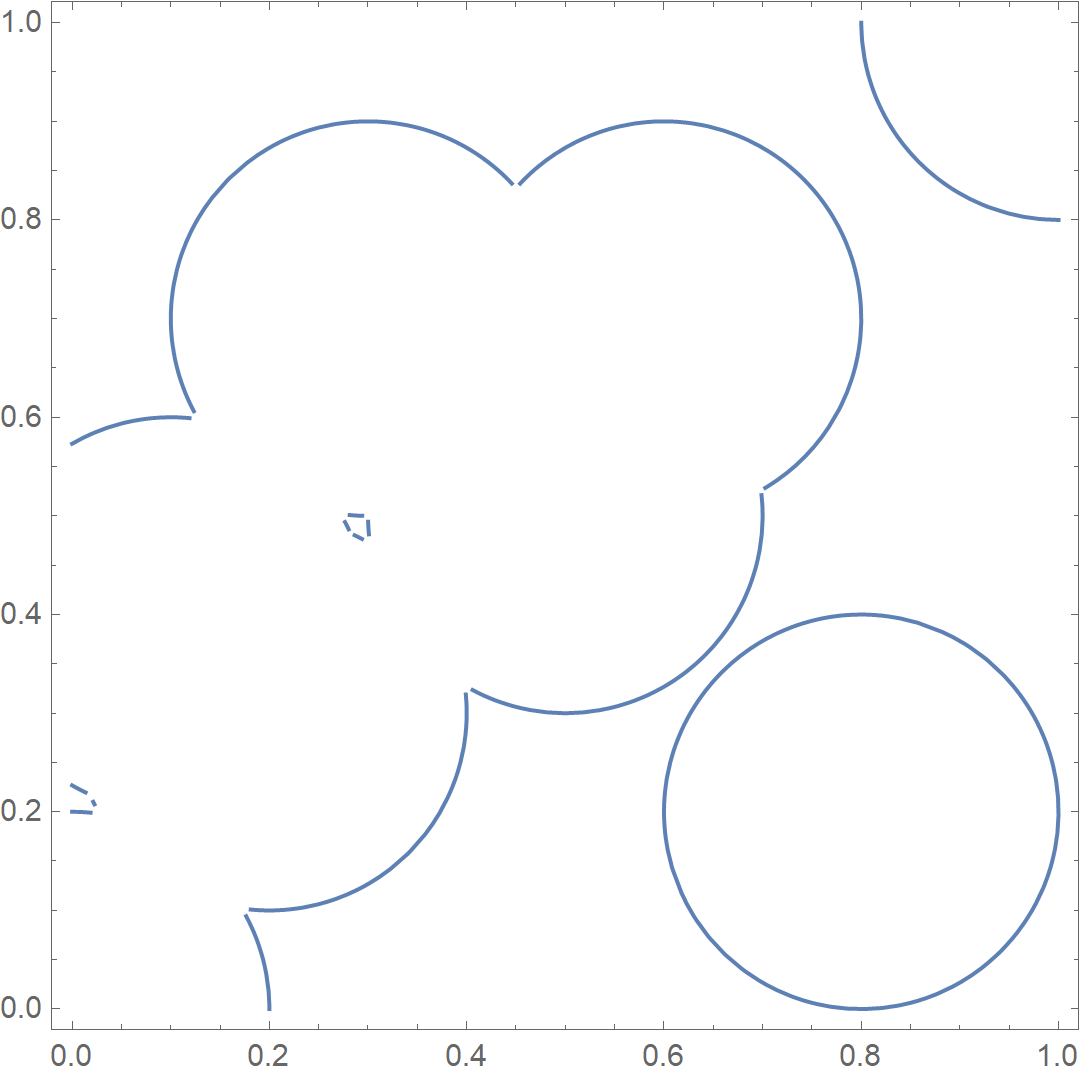}
     \end{subfigure}
     \hfill
     \begin{subfigure}[b]{0.3\textwidth}
         \centering
         \includegraphics[width=\textwidth]{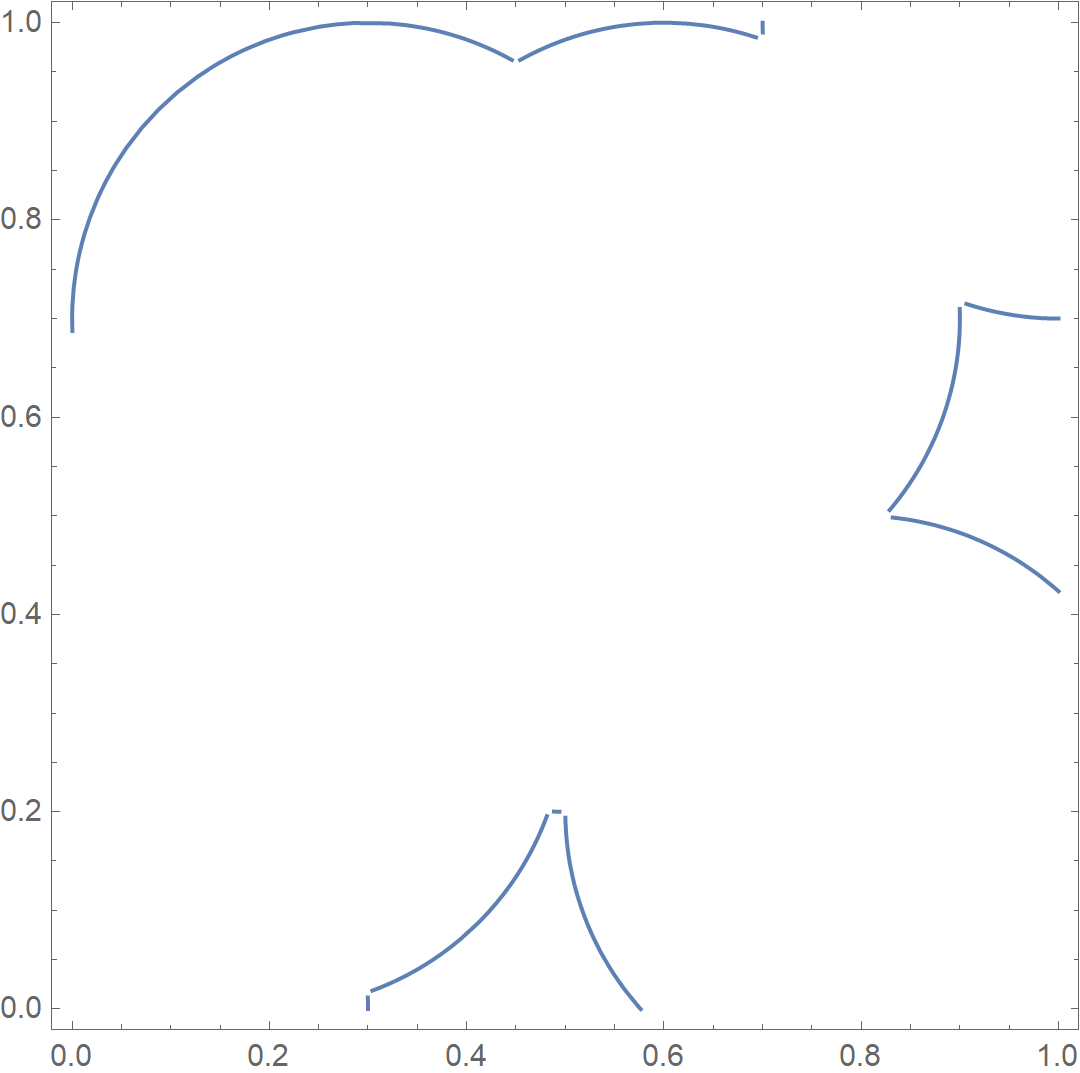}
     \end{subfigure}
        \caption{Examples of distance spheres $S_K(r)$ for a fixed finite set $K\subseteq [0,1]^2$ and three different values of $r$.}
        \label{fig:finite}
\end{figure}

\begin{figure}
     \centering
     \begin{subfigure}[b]{0.3\textwidth}
         \centering
         \includegraphics[width=\textwidth]{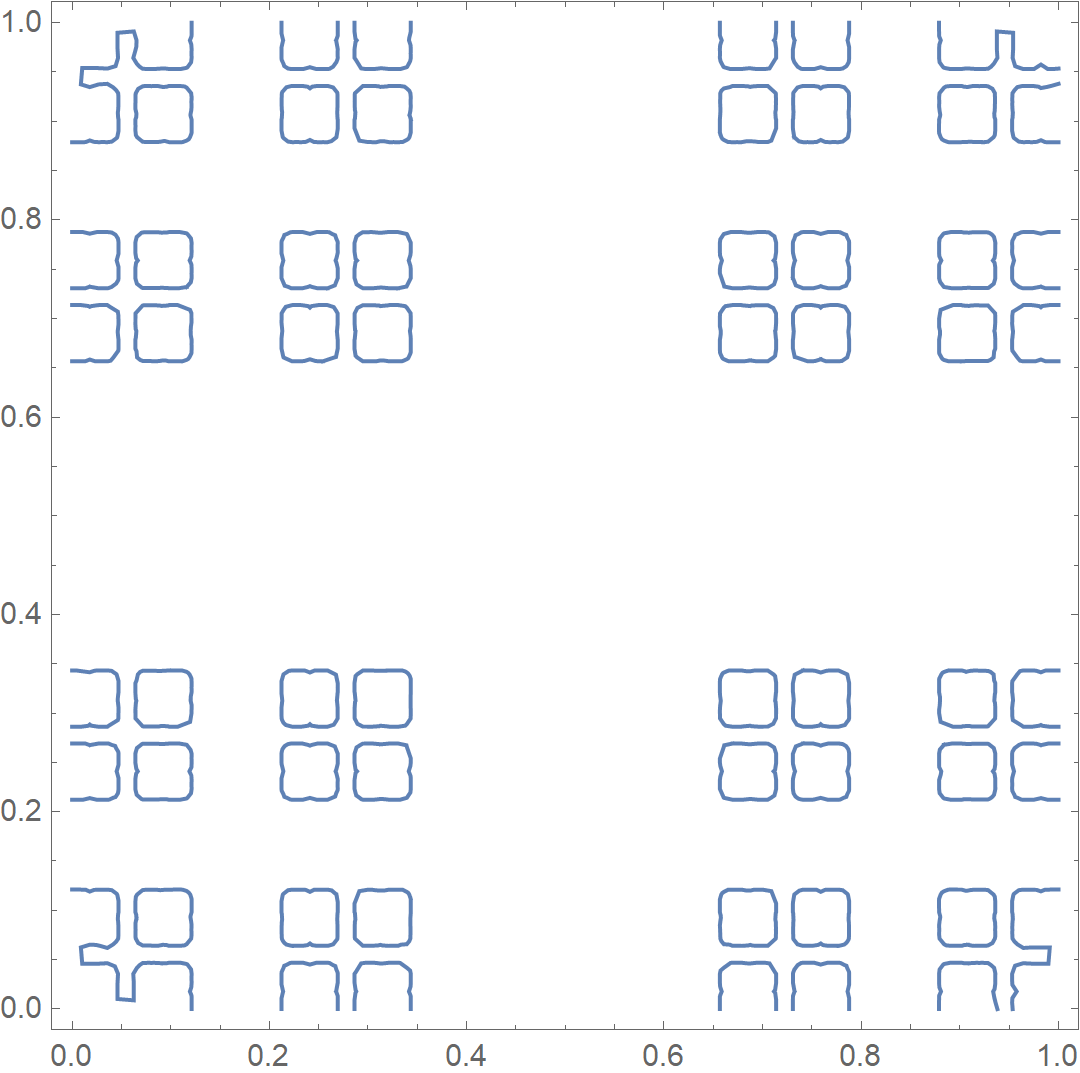}
     \end{subfigure}
     \hfill
     \begin{subfigure}[b]{0.3\textwidth}
         \centering
         \includegraphics[width=\textwidth]{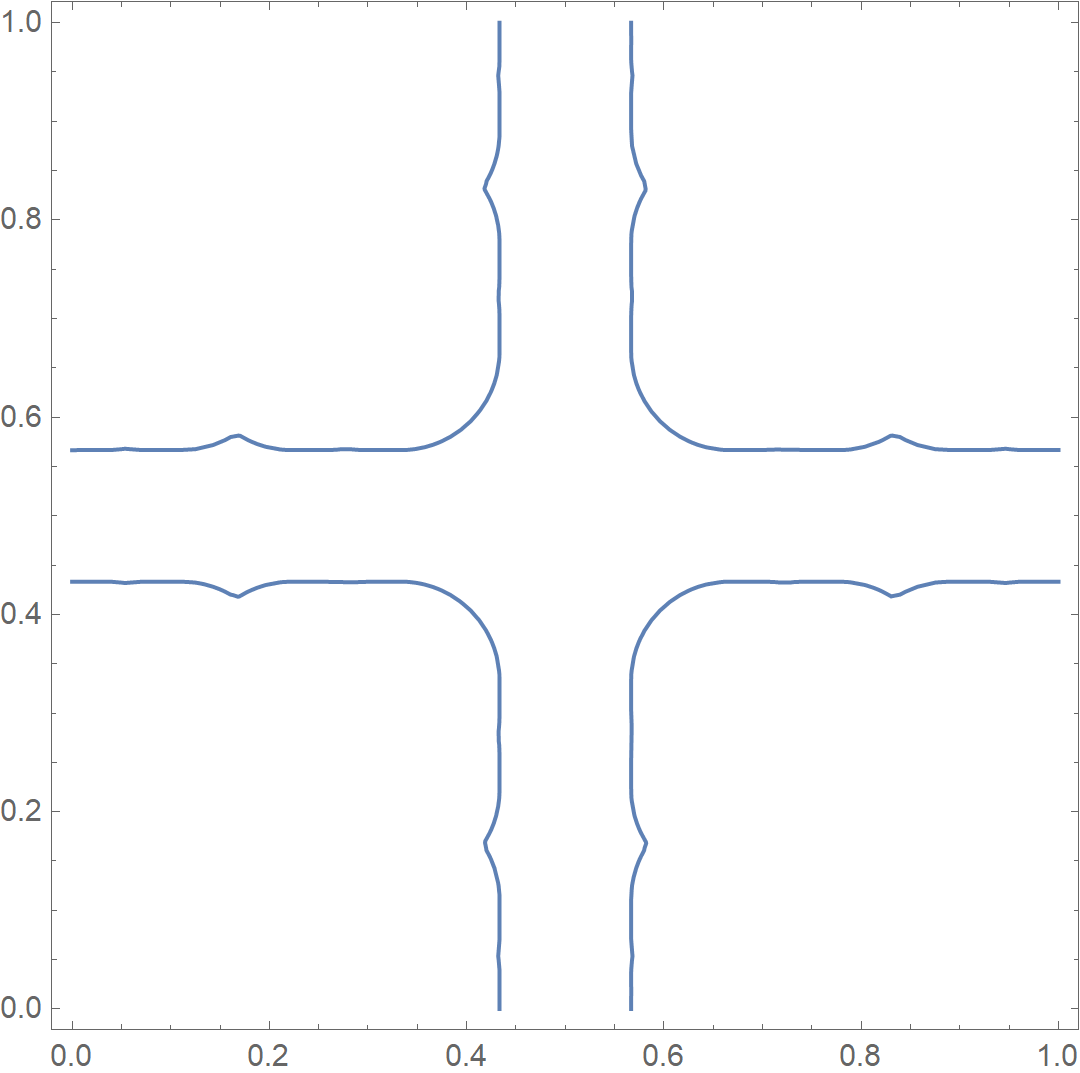}
     \end{subfigure}
     \hfill
     \begin{subfigure}[b]{0.3\textwidth}
         \centering
         \includegraphics[width=\textwidth]{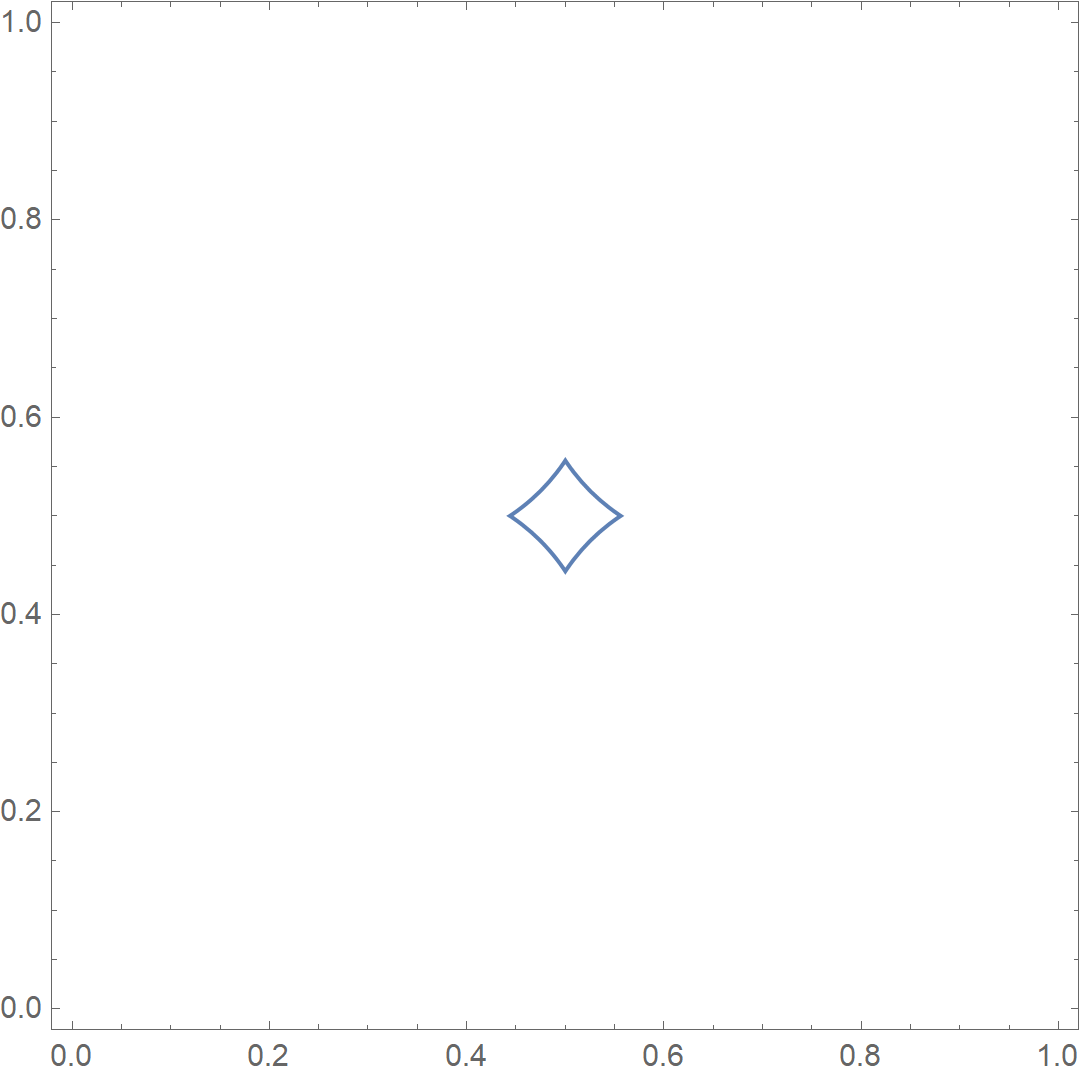}
     \end{subfigure}
        \caption{Examples of distance spheres $S_K(r)$ for three different values of $r$ and a fixed set $K\subseteq [0,1]^2$ that is an approximation of a Cantor set.}
        \label{fig:cantor}
\end{figure}

\begin{definition}
Let $K\subseteq [0,1]^d$ be a set. A set $E\subseteq [0,1]^d$ is called \textit{$K$-straightenable} if there is a bi-Lipschitz map $$g\colon\RR^d\rightarrow\RR^d$$
and an injective function
$$ \phi\colon \{r\geq 0: S_K(r)\cap E \neq \emptyset\} \rightarrow \RR$$
such that
\begin{equation}\label{eq:sc1}
g(S_K(r) \cap E) = \left(\{\phi(r)\} \times \RR^{d-1}\right) \cap g(E) \text{ for all } r\text{ such that } S_K(r) \cap E \neq \emptyset.
\end{equation}
\end{definition}

In other words, $g$ simultaneously ``straightens'' all the sets $S_K(r)\cap E$ into (subsets of) distinct vertical $(d-1)$-dimensional planes.

\begin{example}
If $K=\{\left(0,0\right)\}\subseteq [0,1]^2$, then the set
$$ E = \{ (x,y) \in [0,1]^2 : \frac{1}{2} \leq \sqrt{x^2+y^2} \leq 1 \}$$
is an example of a $K$-straightenable set. (See Figure \ref{fig:example}.) Since $K$ is a single point, the distance spheres $S_K(r)$ are simply arcs of circles. The map
$$ g(x,y) = (\sqrt{x^2+y^2}, \arctan(y/x)),$$
i.e., the map that converts rectangular to polar coordinates, straightens out the distance spheres $S_K(r) \cap E$ into distinct vertical line segments $\left(\{\phi(r)\}\times \RR\right) \cap g(E)$, where we simply take $\phi(r)=r$. One can show that $g$ is bi-Lipschitz on $E$ and extends to a bi-Lipschitz map from $\RR^2$ to $\RR^2$. Note that, while in this example $E$ is the closure of a simple open domain, we do not require this in general.

\begin{figure}
     \centering
     \begin{subfigure}[b]{0.4\textwidth}
         \centering
         \includegraphics[width=\textwidth]{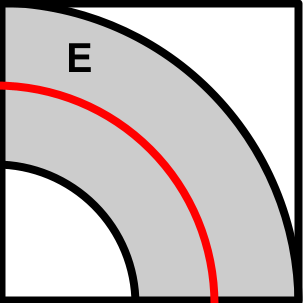}
         \caption{The set $E$ with a marked distance sphere (circle) $S_k(r)$ in red.}
     \end{subfigure}
     \hfill
     \begin{subfigure}[b]{0.4\textwidth}
         \centering
         \includegraphics[width=\textwidth]{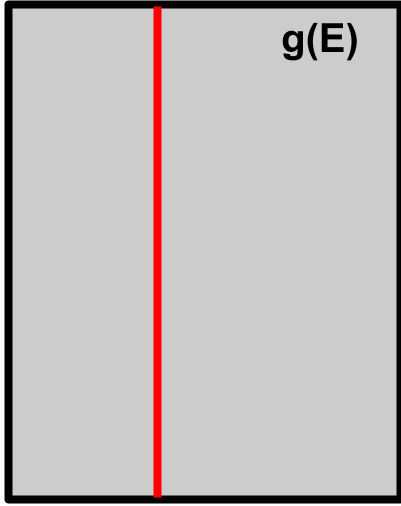}
           \caption{The set $g(E)$ with straightened $g(S_k(r))$ in red.}
     \end{subfigure}

        \caption{A simple example of a straightenable set when $K$ is the one-point set $\{(0,0)\}$.}
        \label{fig:example}
\end{figure}

\end{example}

\begin{definition}
Let $K\subseteq [0,1]^d$ be a set and $\epsilon>0$. We define
$$ \QQ(K,\epsilon) = \{ \text{ dyadic cubes } Q: N_{\epsilon\side(Q)}(K\cap Q) \supseteq Q\}$$
and
$$ D_\epsilon(K) := \cup_{Q\in\QQ(K,\epsilon)} Q.$$
Here $N_\eta(E)$ refers to the open $\eta$-neighborhood of a set $E$; see section \ref{sec:prelims}. In other words, $D_\epsilon(K)$ is the union of all dyadic cubes $Q$ in which $K\cap Q$ is $\epsilon\side(Q)$-dense.
\end{definition}

\begin{theorem}\label{thm:main}
Let $K\subseteq [0,1]^d$ be a set and $\epsilon>0$. Then we can write
$$ [0,1]^d  = E_1 \cup \dots \cup E_M \cup D_{\epsilon}(K) \cup G,$$
where each $E_i$ is $K$-straightenable and $|G|<\epsilon$. 

Moreover, the number of straightenable sets $M$ and the associated bi-Lipschitz constants depend only on $\epsilon$ and $d$. In particular, they do not depend on the set $K$.
\end{theorem}
In this result, $|G|$ refers to the $d$-dimensional volume (Lebesgue measure) of the set $G$; see section \ref{sec:prelims} for notation.

We emphasize that a large part of our interest in Theorem \ref{thm:main} lies in the fact that, in our decomposition, the number of straightenable sets and their associated constants are independent of the starting set $K$.

While Theorem \ref{thm:main} applies to arbitrary sets $K\subseteq [0,1]^d$, we also prove a stronger corollary for a specific class of sets known as \textit{porous sets}. A set $K\subseteq \RR^d$ is \textit{porous} if there is a constant $c>0$ such that, for each $r>0$ and $p\in \RR^d$, the ball $B(p,r)$ contains a ball $B(q,cr)$ that is disjoint from $K$. Many classical fractals, such as the Cantor set and Sierpi\'nski carpet, are porous. More discussion of porous sets can be found, e.g., in \cite[Ch. 5]{TysonMackay}.

If $K$ is a porous set, then we can decompose the entirety of $[0,1]^d$, outside of a set of small measure, into $K$-straightenable sets:

\begin{corollary}\label{cor:porous}
Let $K\subseteq [0,1]^d$ be a porous set with constant $c$. Let $0<\epsilon<c/2$.  Then we can write
$$ [0,1]^d = E_1 \cup \dots \cup E_M \cup G,$$
where each $E_i$ is $K$-straightenable and $|G|<\epsilon$. 

The number of straightenable sets $M$ and the associated bi-Lipschitz constants depend only on $\epsilon$ and $d$, and not on the set $K$.
\end{corollary}

\subsection*{Acknowledgments}
The first named author would like to thank Raanan Schul for helpful conversations at an early state of this project.

\section{Notation and preliminaries}\label{sec:prelims}

\subsection{Basics}
We use the following basic definitions. A function $f$ from a metric space $(X,d_X)$ to a metric space $(Y,d_Y)$ is called \textit{Lipschitz} (or \textit{$L$-Lipschitz} to emphasize the constant) if there is a constant $L$ such that
$$ d_Y(f(x), f(x')) \leq Ld_X(x,x') \text{ for all } x,x'\in X.$$
It is called \textit{bi-Lipschitz} (or \textit{$L$-bi-Lipschitz}) if
$$ L^{-1}d_X(x, x') \leq d_Y(f(x), f(x')) \leq Ld_X(x,x') \text{ for all } x,x'\in X.$$

We use $B(x,r)$ to denote an open ball of radius $r$ centered at $x$ in a metric space, and $\overline{B}(x,r)$ for the corresponding closed ball.

The distance from a point $p$ to a set $K$ in $\RR^d$ is defined as
$$ \dist(p, K) := \inf\{|p-q| : q\in K\}.$$
If $K$ is a set in $\RR^d$ and $\eta>0$, then $N_\eta(K)$ is the open $\eta$-neighborhood of $K$, defined as
$$ N_\eta(K) = \{p\in \RR^d : \dist(p,K)<\eta.\}$$

In $\RR^d$, we will also use the collection of \textit{dyadic cubes}. These consist of all cubes $Q$ in $\RR^d$ of the form
$$ [a_1 2^n, (a_1+1)2^n] \times \dots \times [a_d 2^n, (a_d+1)2^n], $$
where $a_1, \dots, a_d$ and $n$ are integers.

\subsection{Measure, Hausdorff content, and mapping content}

We use $|E|$ to denote the $d$-dimensional volume (Lebesgue measure) of a set in $\RR^d$.

\begin{definition}
Let $E$ be a subset of a metric space $X$, and $k\geq 0$. The \textit{$k$-dimensional Hausdorff content of $E$} is defined by
$$ \HH^k_\infty(E) = \inf_{\mathcal{B}} \sum_{B\in\mathcal{B}} \diam(B)^k,$$
where the infimum is taken over all finite or countable collections of closed balls $\mathcal{B}$ whose union contains $E$.
\end{definition}

The following definition appears first in \cite{AzzamSchul}.
\begin{definition}
Let $f\colon [0,1]^{n+m}\rightarrow Y$ be a function into a metric space, and let $A\subseteq [0,1]^{n+m}$. The \textit{$(n,m)$-mapping content of $f$ on $A$} is:
$$ \HH^{n,m}_\infty(f,A) = \inf_{\mathcal{Q}} \sum_{Q\in \mathcal{Q}} \HH^n_\infty(f(Q))\side(Q)^m,$$
where the infimum is taken over all collections of dyadic cubes $\mathcal{Q}$ in $[0,1]^{n+m}$ whose union contains $A$.
\end{definition}

\subsection{Hard Sard sets}

The following definition was first introduced in \cite{AzzamSchul}. We present the slightly altered version from \cite[Definition 1.3]{DavidSchul}.
\begin{definition}\label{def:HSpair}
Let $n,m\geq 0$. Let $E\subseteq Q_0=[0,1]^{n+m}$ and $f\colon Q_0 \rightarrow X$ a Lipschitz mapping into a metric space. 

We call $E$ a \textbf{Hard Sard set for $f$} if there is a  constant $C_{Lip}$ and a $C_{Lip}$-bi-Lipschitz mapping $g\colon \RR^d \rightarrow \RR^d$ such that the following conditions hold.

Write $\RR^{n+m}=\RR^n \times \RR^m$ in the standard way, and points of $\RR^{n+m}$ as $(x,y)$ with $x\in \RR^n$ and $y\in \RR^m$. Let $F = f \circ g^{-1}$.

We ask that:
\begin{enumerate}[(i)]
\item\label{HS3} If $(x,y)$ and $(x',y')$ are in $g(E)$, then $F(x,y) = F(x',y')$ if and only if $x=x'$. Equivalently,
$$ F^{-1}(F(x,y)) \cap g(E) = (\{x\} \times \RR^m) \cap g(E)$$
\item\label{HS4} The map
$$(x,y) \mapsto (F(x,y),y)$$
is $C_{Lip}$-bi-Lipschitz on the set $g(E)$. 
\end{enumerate}
\end{definition}
Only condition (i) of the definition of Hard Sard set will play a role in this paper.

A slightly simplified version of the main theorem of \cite{DavidSchul} is the following:
\begin{theorem}\label{thm:davidschul}
Let $Q_0$ be the unit cube in $\RR^{n+m}$ and let $f\colon Q_0\rightarrow \RR^n$ be a $1$-Lipschitz map.

Given any $\gamma>0$, we can write
$$ Q_0 = E_1 \cup \dots \cup E_M \cup G,$$
where $E_i$ are Hard Sard sets and
$$ \HH^{n,m}_\infty(f,G) < \gamma.$$
The constant $M$ and the constants $C_{Lip}$ associated to the Hard Sard sets $E_i$ depend only on $n$, $m$, and $\gamma$.
\end{theorem}

\section{Lemmas}

\begin{lemma}\label{lem:distancefunction}
If $K$ is any set in $\RR^d$, the function 
$$ f(x) = \dist(x,K)$$
is $1$-Lipschitz.
\end{lemma}
\begin{proof}
Let $x,y \in \mathbb{R}^d$, and let $K \subseteq \mathbb{R}^d$. Without loss of generality, assume $f(x) \geq f(y)$. Let $z_y$ be a point in the closure of $K$ such that $\inf\{|y-z|:z \in K\} = |y-z_y|.$ Then $f(y) = \dist(y,K) = |y-z_y|$. Then applying the triangle inequality, we have $$\dist(x,K) = \inf\{|x-z|:z \in K\} \leq |x-z_y| \leq |x-y| + |y-z_y| = |x-y| + \dist(y,K).$$ Then $$\dist(x,K) - \dist(y,K) \leq |x-y|.$$ Thus, since $f(x) \geq f(y)$, $$|f(x) - f(y)| = |\dist(x,K) - \dist(y,K)| = \dist(x,K) - \dist(y,K) \leq |x-y|,$$ and so $f(x) = \dist(x,K)$ is $1$-Lipschitz.
\end{proof}

\begin{lemma}\label{lem:intervalcontent}
If $[a,b]$ is a compact interval in $\RR$, then $\HH^1_\infty([a,b])=b-a$.
\end{lemma}
\begin{proof}
Notice that a closed ball in $\RR$ is just a closed interval $[a_{i},b_{i}]$. Then for an interval $[a,b]$, we have
$$\overline{B}\left(\frac{a+b}{2},\frac{b-a}{2}\right)=[a,b],$$
which implies $\HH^1_ \infty([a,b])\le \diam(\overline{B}(\frac{a+b}{2},\frac{b-a}{2}))=b-a$.

Now let $\{\overline{B}_{i} = [a_i,b_i]\}$ be a collection of closed balls that cover the interval $[a,b]$. Then 
$$\displaystyle \sum_{i}\diam(\overline{B}_{i})=\displaystyle \sum_{i}\diam([a_{i},b_{i}])\ge b-a,$$
where the inequality is a basic fact in measure theory. Taking the infimum of both sides we get $\HH^1_ \infty([a,b])\ge b-a$. Hence, $\HH^1_ \infty([a,b])=b-a$, as desired.  
\end{proof}

Now fix $K\subseteq [0,1]^d$. Let $f(x) = \dist(x,K)$. 

\begin{lemma}\label{lem:segment}
Let $x\in [0,1]^d$ and $z\in \overline{K}$ such that
$$ f(x) = |z-x|.$$
If $y$ is a point on the line segment from $x$ to $z$, then
$$ |f(y)-f(x)| = |y-x|$$
\end{lemma}
\begin{proof}
By Lemma \ref{lem:distancefunction}, we know $f(x)=\dist(x,K)$ is 1-Lipschitz. Then, $|f(x)-f(y)|\le |x-y|$. However,
$$|f(x)-f(y)|=|\dist(x,K)-\dist(y,K)|=\dist(x,K)-\dist(y,K)\ge |x-z|-|y-z|,$$
as $\dist(y,K)=\inf\{|y-z|:z\in K\}$. Then, 
$$\dist(x,z)-\dist(y,z)=|x-z|-|y-z|=|x-y|,$$
as $y$ is on the line segment from $x$ to $z$. Thus, $|f(x)-f(y)|=|f(y)-f(x)|=|y-x|.$ 
\end{proof}
\begin{lemma} \label{cubelemma}
Let $\delta>0$ and let $Q$ be a dyadic cube in $\RR^d$ such that
$$ \HH^1_\infty(f(Q)) < \delta \side(Q).$$
Then $Q\in \QQ(K,c_d\delta)$, where $c_d=\sqrt{d}+1$.
\end{lemma}

\begin{proof}
Let $\delta>0$ and let $Q$ be a dyadic cube in $\RR^d$ such that $\HH^1_\infty(f(Q)) < \delta \side(Q)$. Let $Q' \subseteq Q$ be the set of points $x$ in $Q$ such that $\dist(x,\partial Q) \geq \delta\side(Q)$, where $\partial Q$ is the set of boundary points of $Q$.

\begin{claim}\label{claim:point} Let $x \in Q'$. Then there must be a point of $K$ inside the ball $B(x, \delta \side(Q))\subseteq Q$.
\end{claim}

\begin{proof}[Proof of Claim \ref{claim:point}] Let $x \in Q'$, and let $z'$ be a point in the closure of $K$ such that $$f(x) = \dist(x,K) = |x-z'|.$$

If $z'$ is not in $Q$, then let $S$ be the line segment from $z'$ to $x$, and let $y$ be the point on the boundary of $Q$ such that $y \in S$. Then by Lemma \ref{lem:segment}, $$|f(y) - f(x)| = |y - x| \geq \delta\side(Q).$$ 

Now since $Q$ is closed and bounded, it is compact. Also, since $Q$ is convex, it is connected. Then since $f(x) = \dist(x,K)$ is continuous, $f(Q) \subseteq \RR$ is also compact and connected. Then $f(Q) = [a,b]$ for some $a \leq b$. Then by Lemma \ref{lem:intervalcontent}, $$\HH^1_\infty(f(Q)) = \HH^1_\infty([a,b]) = b - a.$$ 

Then we have $$\HH^1_\infty(f(Q)) = b - a \geq |f(y) - f(x)| \geq \delta\side(Q).$$ This contradicts the assumption that $\HH^1_\infty(f(Q)) < \delta\side(Q)$. Thus it must be that $z'$ is in $Q$. Then suppose for the sake of contradiction that $z'$ is not contained in $B(x, \delta \side(Q))$. Then $$f(x) = \dist(x,K) = |x-z'| \geq \delta\side(Q),$$ which leads us to the same contradiction as above. Thus it must be that $z'$ is contained in $B(x, \delta\side(Q))$. Since $z'$ is in the closure of $K$, $B(x,\delta\side(Q))$ must contain a point of $K$.
\end{proof}
Thus for any $x \in Q'$, there is a point $z$ of $K$ inside $B(x, \delta \side(Q))$, and so $$|x-z| < \delta\side(Q) < c_d\delta\side(Q)$$.

Now consider $x \in Q$ such that $x \not\in Q'$. Then there is some $x' \in Q'$ such that $|x-x'|\leq \sqrt{d}\delta\side(Q).$ Since $x' \in Q'$, there is some $z \in K$ such that $z \in B(x', \delta\side(Q))$. Then 
$$|x-z| \leq |x-x'| + |x'-z|< \sqrt{d}\delta\side(Q) + \delta\side(Q) < c_d\delta\side(Q).$$
Thus for any $x \in Q$, there exists $z \in K \cap Q$ such that $|x-z| < c_d\delta\side(Q)$, and so $Q \in \QQ(K,c_d\delta)$.  

\end{proof}

The last lemma concerns the concept of mapping content $\HH^{n,m}_\infty$ defined above.

\begin{lemma} \label{rewritelemma}
Let $f:Q_0 \rightarrow X$ be $1$-Lipschitz and $n,m\geq 1$. Let $A\subseteq Q_0$ and suppose
$$ \HH^{n,m}_\infty(f,A) < \delta$$
Then we can write
$$ A \subseteq A' \cup \bigcup_{i} Q_i,$$
where 
\begin{enumerate}[(i)]
\item $|A'|< \sqrt{\delta}$, 
\item $Q_i$ are dyadic cubes,
\item $\HH^n_\infty(f(Q_i)) < \sqrt{\delta}\side(Q_i)^n$ for each $i$.
\end{enumerate}
\end{lemma}
\begin{proof}
We have $$\HH^{n,m}_\infty(f,A) = \inf_{\mathcal{Q}} \sum_{Q\in \mathcal{Q}} \HH^n_\infty(f(Q))\side(Q)^m < \delta,$$ where the infimum is taken over all collections of dyadic cubes $\mathcal{Q}$ in $Q_0$ whose union contains $A$. By definition of infimum, there exists a collection of dyadic cubes $\mathcal{R} = \{R_j\}_{j\in J}$, whose union contains $A$, such that
\begin{equation}\label{eq:Rj}
\sum_{j\in J}\HH^n_\infty(f(R_j))\side(R_j)^m < \delta.
\end{equation}
We split these cubes $R_j$ into two collections:
$$\mathcal{R}^1 = \{R_j \in \mathcal{R} : \HH^n_\infty(f(R_j)) < \sqrt{\delta}\side(R_j)^n\},$$
and
$$\mathcal{R}^2 = \{R_j \in \mathcal{R} : \HH^n_\infty(f(R_j)) \geq \sqrt{\delta}\side(R_j)^n\}.$$ $\mathcal{R}^1$ will become our collection of dyadic cubes $\{Q_i\}$. The union of cubes in $\mathcal{R}^2$ will be our set $A'$, so we want to show that $\Big|\bigcup_{R_j \in R^2}R_j\Big| < \sqrt{\delta}$. 

Let $J_2 := \{j \in J : R_j \in \mathcal{R}^2\}$. Then, using \eqref{eq:Rj}, we have $$\delta > \sum_{j \in J}\HH^n_\infty(f(R_j))\side(R_j)^m \geq \sum_{j \in J_2}\HH^n_\infty(f(R_j))\side(R_j)^m.$$ Then by the definition of our set $\mathcal{R}^2$, we have $$\delta >  \sum_{j \in J_2}\HH^n_\infty(f(R_j))\side(R_j)^m \geq \sum_{j \in J_2}\sqrt{\delta}\side(R_j)^n \side(R_j)^m =  \sqrt{\delta}\sum_{j \in J_2}|R_j| \geq \sqrt{\delta}\Big|\bigcup_{j \in J_2}R_j\Big|.$$

Thus we have $$\Big|\bigcup_{j \in J_2}R_j\Big| < \sqrt{\delta}.$$ Therefore, if we define $A'$ to be the union of the cubes in $\mathcal{R}^2$ and define $\{Q_i\}$ to be the collection of cubes in $\mathcal{R}^1$, then we can write $$A \subseteq A' \cup \bigcup_i Q_i,$$ where properties (i)-(iii) hold for $A'$ and each $Q_i$. 
\end{proof}
\section{Proofs of the main results}

\begin{proof}[Proof of Theorem \ref{thm:main}]
Take $K\subseteq[0,1]^{d}$ and $\epsilon>0$. Let $f(x)=\dist({x,K})$, which is $1$-Lipschitz by Lemma \ref{lem:distancefunction}. Applying Theorem \ref{thm:davidschul} to $f$ with $n=1$, $m=d-1$, and $\gamma=\frac{\epsilon^2}{c_d^2}$ (where $c_d=\sqrt{d}+1$ as in Lemma \ref{cubelemma}) we have that $$[0,1]^{d}=E_{1}\cup...\cup E_{M}\cup G_{0},$$
where $E_{i}$ are Hard Sard sets for $f$ and $\HH^{1,d-1}_\infty(f,G_{0})<\frac{\epsilon^2}{c_d^2}$. 

The following two claims combine to complete the proof of Theorem \ref{thm:main}.

\begin{claim}\label{claim:straight}
The Hard Sard sets $E_{i}$ are $K$-straightenable sets.
\end{claim}
\begin{proof}[Proof of Claim \ref{claim:straight}]
Throughout this proof, we write points of $\RR^d$ as $(x,y)$, where $x\in \RR$ and $y\in\RR^{d-1}$. Let $E=E_i$ for some $i\in\{1, \dots, M\}$.

By Definition \ref{def:HSpair}, there is a bi-Lipschitz map $g:\RR^d \to \RR^d$ such that if $F=f\circ g^{-1}$, then for $(x,y),(x^{'},y^{'})\in g(E)$, $F(x,y)=F(x^{'},y^{'})$ if and only if $x=x^{'}$.

Now, take any $r$ such that $S_{K}(r)$ intersects $E$ and consider any point $p\in S_{K}(r)\cap g(E)$. Then $g(p)=(x,y)$, and if any other point $q\in S_{k}(r)\cap g(E)$, then $f(p)=f(q)=r$, which implies that $F(g(q))=F(g(p))$. Then, $g(q)\in (\{x\}\times\RR^{d-1})\cap g(E)$. Hence, $g(S_{K}(r))\cap g(E)\subseteq (\{x\}\times\RR^{d-1})\cap g(E)$.

Now take any point $(x,y')\in (\{x\}\times\RR^{d-1})\cap g(E)$, where $(x,y^{'})=g(p')$, for some $p'\in E$. Then $F(x,y')=F(x,y)=F(g(p))=f(p)=r=f(p')=\dist(p^{'},K)$. Thus $p'\in S_{K}(r)\cap g(E)$, and it follows that $(x,y')=g(p')\in g(S_{K}(r))\cap g(E)$. Thus $(\{x\}\times\RR^{d-1})\cap g(E)\subseteq g(S_{K}(r))\cap g(E)$, yielding the desired equality.

Lastly, define $$\phi:\{r\ge0:S_{K}(r)\cup E\neq \emptyset\}\to \RR$$
so that $\phi(r)$ is equal to the first coordinate $x$ of all points $(x,y)\in g(S_{K}(r)\cap E)$. (Note that all such points share a common first coordinate by our work above.)

By definition, $g(S_{K}(r)\cap E)=(\{\phi(r)\}\times\RR^{d-1})\cap g(E)$. Now suppose $\phi(r)=\phi(r^{'})$. Then there are points $p\in S_{k}(r)\cap E$ and $p^{'}\in S_{K}(r^{'})\cap E$ such that $g(p)=(x,y)=g(p^{'})$. Thus $F(g(p))=F(g(p^{'}))$, which implies that $f(p)=f(p^{'})$, and therefore $r=r^{'}$. Hence, $\phi$ is injective. 
\end{proof}

\begin{claim}\label{claim:garbage}
The set $G_0$ is contained in $G\cup D_\epsilon(K)$, where $G$ is a subset of $[0,1]^d$ with $|G|<\epsilon$.
\end{claim}
\begin{proof}[Proof of Claim \ref{claim:garbage}]
Applying Lemma \ref{rewritelemma}, we can write $$G_0 \subseteq G \cup \bigcup_i Q_i,$$ where $|G| < \frac{\epsilon}{c_d}$ and $Q_i$ are dyadic cubes with $\HH_\infty^1(f(Q_i)) < \frac{\epsilon}{c_d} \side (Q_i)$ for each $i$. Then by Lemma \ref{cubelemma}, we have $Q_i \in \QQ(K,\epsilon)$ for every $i$, and so $Q_i \in \QQ(K,\epsilon)$ for all $i$. Then we have $$\bigcup_i Q_i \subseteq D_\epsilon(K),$$ and so $$G_0 \subseteq G \cup \bigcup_i Q_i \subseteq G \cup D_\epsilon(K),$$ where $|G| < \frac{\epsilon}{c_d} < \epsilon$  
\end{proof}

\end{proof}

\begin{proof}[Proof of Corollary \ref{cor:porous}]
Let $K \subseteq [0,1]^d$ be a porous set with constant $c$, and let $0 < \epsilon < c/2$. Let $Q$ be any dyadic cube in $[0,1]^d$. Let $p$ be the point in the center of $Q$, and let $r := \frac{1}{2}\side(Q)$. Consider the ball $B(p,r) \subseteq Q$. Since $K$ is porous, there exists some point $q$ in $B(p,r)$ such that $$B(q,cr) \subseteq B(p,r)$$ and $$B(q,c r) \cap K = \emptyset.$$ Then for every $z \in K$, $$|q-z| \geq cr = \frac{c}{2}\side(Q) > \epsilon\side(Q).$$ Thus $Q \notin \QQ(K,\epsilon)$. Since this is true for every dyadic cube in $[0,1]^d$, $$\QQ(K,\epsilon) = \emptyset,$$ and so $$D_\epsilon(K) = \emptyset.$$ Then by Theorem \ref{thm:main}, we can write $$[0,1]^d = E_1 \cup ...\cup E_M \cup D_\epsilon (K) \cup G = E_1 \cup ...\cup E_M \cup G,$$ where each $E_i$ is $K$-straightenable, $|G|<\epsilon$ and the number of straightenable sets $M$ and the associated bi-Lipschitz constants depend only on $\epsilon$ and $d$, and not on the set $K$.   
\end{proof}

\bibliographystyle{plain} 
\bibliography{distancespherebib}

\end{document}